\documentclass{amsart}

 \newtheorem{thm}{Theorem}[section]

 \newtheorem{lemma}[thm]{Lemma}

\theoremstyle{definition}
\newtheorem{defn}[thm]{Definition}

\newtheorem{question}{Question}

\newcommand{\A}{\mathcal A}
\newcommand{\B}{\mathcal B}
\newcommand{\Ll}{\mathcal{L}}

\newcommand{\K}{\mathcal{K}}
 
 \newcommand{\C}{\mathbb{C}}

 \newcommand{\al}{\alpha}

\newcommand{\de}{\delta}
 \newcommand{\be}{\beta}
 \newcommand{\As}{\A^{\#}}

 \newcommand{\hot}{\hat{\otimes}}

\begin{document}
Banach algebras 2009, 441-454, Banach Center Publ., 91, Polish Acad. Sci. Inst. Math., Warsaw, 2010.

\title[Generalized amenability]{Solved and unsolved problems in generalized notions of amenability for Banach algebras}


\author{Yong Zhang}
\address{Department of Mathematics, University of Manitoba\\
Winnipeg MB, R3T 2N2 Canada\\
E-mail: zhangy@cc.umanitoba.ca}

\thanks{Supported by NSERC grant 238949-2005.}

\begin{abstract}
We survey the recent investigations on (bounded, sequential) approximate amenability/contractibility and pseudo-amenability/contractibility for Banach algebras. We will discuss the core problems concerning these notions and address the significance of any solutions to them to the development of the field. A few new results are also included.
\end{abstract}

\maketitle

\section{Introduction}

The concept of amenability for Banach algebras was introduced by B. E. Johnson in 1972. Since his groundwork  \cite{Johns_M} was published, the notion has proved to be of enormous importance in the theory of Banach algebras, operator algebras and abstract harmonic analysis. It reflects intrinsic features of many types of Banach algebras. For example, The group algebra $L^1(G)$ on a locally compact group $G$ is amenable if and only if $G$ is an amenable group \cite{Johns_M}; the Fourier algebra $A(G)$ is amenable if and only if $G$ has an abelian subgroup of finite index \cite{LLW, FR}; a uniform algebra is amenable if and only if it is isomorphic to $C_0(X)$ for a locally compact Hausdorff space $X$ \cite{se_uniform}; a C*-algebra is amenable if and only if it is nuclear \cite{Con, haage}. However, it has been also realized that in many instances amenability is too restrictive. It is essentially a kind of finiteness condition on a Banach algebra. Many efforts have been made in the literature to extend or to modify the concept of amenability. Weak amenability was introduced in \cite{BCD}; $n$-weak amenability was introduced in \cite{D-G-G}; Operator amenability was introduced in \cite{Ruan} and Connes amenability in \cite{HEL_C,JKR,RUN}. In this survey paper we will focus on generalized amenability for Banach algebras, discussing solved and unsolved problems in this recently developed field.

Let $\A$ be a Banach algebra and let $X$ be a Banach $\A$ bimodule. A linear mapping $D$: $\A \to X$ is a \emph{derivation} if it satisfies $D(ab) = aD(b) + D(a)b$ for  all $a, b\in \A$. Given an $x\in X$, the mapping $ad_x$: $a\mapsto ax-xa$ ($a\in \A$) is a continuous derivation, called an \emph{inner derivation}. The algebra $\A$ is called \emph{contractible} if every continuous derivation $D$: $\A \to X$ is inner for each Banach $\A$ bimodule $X$ \cite{HEL, CAL}. The algebra $\A$ is called \emph{amenable} if every continuous derivation $D$: $\A \to X^*$ is inner for each Banach $\A$ bimodule $X$, where $X^*$ is the dual module of $X$ \cite{Johns_M}. So far the only known contractible Banach algebras are the direct sums of finite full matrix algebras. We call a derivation $D$: $\A \to X$ \emph{approximately inner} if there is a net $(x_i)\subset X$ such that, for each $a\in \A$,
 \begin{equation}\label{approx inner}
 D(a) = \lim_i ad_{x_i}(a) \quad ( \text{i.e. } D(a)=\lim_i ax_i - x_i a) 
\end{equation}
 in the norm topology of $X$.
If in the above definition $(x_i)$ can been chosen so that $(ad_{x_i})$ is bounded as a net of operators from $\A$ into $X$ (note $(x_i)$ is not necessarily bounded in the case), then $D$ is called \emph{boundedly approximately inner}. If $(x_i)$ can been chosen to be a sequence, then $D$ is called \emph{sequentially approximately inner}. In the definition we may require the limit in (\ref{approx inner}) hold in other topologies of $X$. For example, if the convergence of (\ref{approx inner}) is only required in weak topology of $X$, then we call $D$ \emph{weakly approximately inner};
If $X$ is a dual $\A$-module and the convergence of (\ref{approx inner}) is only required in weak* topology of $X$, then we call $D$ \emph{weak* approximately inner}.
If the convergence of (\ref{approx inner}) is uniform in $a$ on the unit ball of $\A$, then we call $D$ \emph{uniformly approximately inner}.

\begin{defn}
A Banach algebra $\A$ is called
(resp. boundedly, sequentially, uniformly or weakly) approximately contractible if every continuous derivation $D$: $\A\to X$ is (resp. boundedly, sequentially, uniformly or weakly) approximately inner for each Banach $\A$-bimodule $X$.
\end{defn}
We will sometimes abbreviate the phrase (boundedly, sequentially, uniformly or weakly) approximately contractible to (bdd., seq., unif. or w.) a. c..

\begin{defn}
A Banach algebra $\A$ is called (resp. boundedly, sequentially, uniformly or weak*) approximately amenable if every continuous derivation $D$: $\A\to X^*$ is (resp. boundedly, sequentially, uniformly or weak*) approximately inner for each Banach $\A$-bimodule $X$.
\end{defn}
We may also use the abbreviated notation (bdd., seq., unif., or w*.) a. a. to denote (boundedly, sequentially, uniformly or weak*) approximately amenable.

These approximate versions of amenability and contractibility were introduced by F. Ghahramani and R. Loy in \cite{GL}. For a Banach algebra $\A$, we may associate a unit $e$ to it to consider its unitization algebra
 $\A^\sharp = \A \oplus \C e$. Clearly,
$\A$ is approximately contractible/amenable in any of the above modes if and only if $\A^\sharp$ is. In section~\ref{concrete} we will see examples of approximately amenable but not amenable Banach algebras.

Given a Banach algebra $\A$, the projective tensor product $\A\hot \A$ is naturally a Banach $\A$ bimodule. The multiplication mapping $\pi$: $\A \hot \A \to \A$ defined by $\pi(a\otimes b)= ab$ ($a,b\in \A$) is a contractive $\A$ bimodule morphism. It is known that $\A$ is contractible if and only if there is $u\in \A\hot\A$ such that $au-ua = 0$ and $\pi(u)a = a$ for all $a\in \A$ \cite{HEL}. 
Such $u$ is called a \emph{diagonal} for $A$. It is also well-known that
 $\A$ is amenable if and only if there is a bounded net $(u_i)\subset \A\hot\A$ such that $au_i-u_ia\to 0$ and $\pi(u_i)a\to a$ for all $a\in \A$ \cite{John_diag}. Such net $(u_i)$ is called a \emph{bounded approximate diagonal} for $\A$. Hence, contractibility and amenability may be defined in terms of the existence of a diagonal and the existence of a bounded approximate diagonal, respectively. These characterizations provide another way to generalize amenability for Banach algebras.
 
\begin{defn}
 A Banach algebra $\A$ is called \emph{pseudo-amenable} (briefly, ps. a.) if it has an approximate diagonal, i.e. if there is a net $(u_i)\subset \A\hot\A$ such that $au_i-u_ia\to 0$ and $\pi(u_i)a\to a$ for all $a\in \A$.
The algebra $\A$ is called \emph{pseudo-contractible} (briefly, ps. c.) if it has a central approximate diagonal, i.e. if there is a net $(u_i)\subset \A\hot\A$ such that $au_i-u_ia = 0$ and $\pi(u_i)a \to a$ for all $a\in \A$.
\end{defn}

The qualifier \emph{bounded} prefixed to the above notions specifies that there is a constant $K>0$ such that the net $(u_i)$ may be chosen so that $\|au_i-u_ia\|\leq K\|a\|$ and $\|\pi(u_i)a\|\leq K \|a\|$ for all $a\in \A$.
The qualifier \emph{sequential} prefixed to the notions will indicate that $(u_i)$ is a sequence. Pseudo-amenability and pseudo-contractibility were introduced by F. Ghahramani and the author in \cite{GZ}.

There are many pseudo-amenable and pseudo-contractible Banach algebras which are not amenable and not even approximately amenable. The simplest example is $\ell^1$ with the pointwise multiplication.

Some of the above generalized versions of amenability are equivalent. Some are the same if the Banach algebra has a bounded approximate identity. In Section~2 we will discuss these relations. In Section~3 we focus on elementary properties of generalized amenability, while in Section~4 we discuss generalized amenability of concrete types of Banach algebras.

\section{Relations}

 It was shown in \cite{GOU} that a Banach algebra $\A$ is amenable if and only if for every Banach $\A$ bimodule $X$ and every continuous derivation $D$: $\A \to X$ there is a bounded net $(x_i)\subset X$ such that $ad_{x_i}$ approaches $D$ in the strong operator topology, i.e. $D(a) = \lim_i ax_i-x_ia$ in the norm topology of $X$ for each $a\in \A$. This implies that every amenable Banach algebra is boundedly approximately contractible. By the principle of uniform boundedness we also see that $\A$ is boundedly approximately amenable/contractible if it is sequentially approximately amenable/contractible. In general, we clearly have the following implications.

\begin{equation}\label{relation 1}
\text{contractible} \Rightarrow\begin{Bmatrix} \text{amenable}\\
                                      \text{seq. a. c.}\end{Bmatrix}\Rightarrow \begin{Bmatrix} \text{seq. a. a.}\\
                                                                                \text{bdd. a. c.}\end{Bmatrix}\Rightarrow\begin{Bmatrix} \text{bdd. a. a.}\\ \text{a. c.}\end{Bmatrix}\Rightarrow
\text{a. a.}\; ,
\end{equation}
\begin{equation}\label{relation 2}
\text{contractible} \Rightarrow \text{unif. a. c.} \; ,\quad \text{and} \quad \text{amenable} \Rightarrow \text{unif. a. a.} 
\end{equation}

It is not trivial but turns out that the converses of the two implications in (\ref{relation 2}) are also true.

\begin{thm}\label{ua is a} Let $\A$ be a Banach algebra.
\begin{enumerate}
\item If $\A$ is uniformly approximately contractible, then it is contractible \cite{GL}.
\item If $\A$ is uniformly approximately amenable, then it is amenable \cite{Pir_approj} \cite{GLZ}.
\end{enumerate}
\end{thm}
In fact, it was left open in \cite{GL} whether or not the second assert of the above theorem was true. The proofs given in \cite{Pir_approj} and \cite{GLZ} are quite different. Regarding the implication chain (\ref{relation 1}), we point out first that, so far, all known approximately amenable Banach algebras are boundedly approximately contractible. The following theorem clarifies some equivalences.

\begin{thm}[\cite{GLZ}]\label{ac is aa}
For a Banach algebra $\A$ the following are equivalent
\begin{enumerate}
\item $\A$ is approximately contractible;
\item $\A$ is approximately amenable;
\item $\A$ is weakly approximately contractible;
\item $\A$ is weak* approximately amenable.
\end{enumerate}
\end{thm}

We may discuss further the equivalence in the cohomology setting.
Let $\A$ be a Banach algebra and $X$ be a Banach $\A$ bimodule. For each integer $n\geq 1$, we denote by $\Ll^n(\A,X)$ the linear space of all bounded $n$-linear functionals from $\A^n$ into $X$. The space $\Ll^n(\A,X)$ may be equipped with various topologies. For example, we may consider the operator uniform norm topology, the strong operator topology or the weak operator topology on $\Ll^n(\A,X)$, and we will denote the result topological vector spaces, respectively, by $\Ll^n_u(\A,X)$, $\Ll^n_s(\A,X)$ and $\Ll^n_w(\A,X)$. Consider the complex
\[  0\overset{\de^0}{\to} X \overset{\de^1}{\to} \Ll^1(\A,X)\overset{\de^2}{\to} \Ll^2(\A,X)\overset{\de^3}{\to}\ldots \overset{\de^n}{\to} \Ll^n(\A,X) \ldots,  \]
where $\de^n$: $\Ll^{n-1}(\A,X)\to \Ll^n(\A,X)$ is the linear mapping (see \cite{Johns_M} or \cite{DAL} for details) defined by
\begin{align*}
 \de^nT\,(a_1,a_2,\ldots,a_n) =& a_1T(a_2,\ldots,a_n)+ \sum_{i=1}^{n-1}{(-1)^iT(a_1,\ldots,a_ia_{i+1},\ldots, a_n)}\\
                               & + (-1)^nT(a_1,\ldots,a_{n-1}) a_n \quad (a_i\in \A, \; i=1, 2, \ldots, n).  
\end{align*}
We will use $\de^n_X$ instead of $\de^n$ if we need to highlight that the ground module in the complex is $X$. Clearly $\de^n$ is continuous if we equip all $\Ll^n(\A,X)$ ($n=1, 2, \ldots$) with the $u$-, $s$-, or $w$-topology. So $\ker(\de^{n+1})$ is closed in $\Ll^n(\A,X)$ in each of these topologies. Hence $cl(\text{Im}\,\de^n)\subset \ker(\de^{n+1})$, where $cl$ denotes the closure in any of the above topologies. Let us focus on the strong operator topology case, and denote the closure in this topology by $cl_s$. Define $H_s^n(\A,X)= \ker(\de^{n+1})/cl_s(\text{Im}\,\de^n)$. Clearly, to say $\A$ being approximately amenable is to say $H_s^1(\A,X^*)=\{0\}$ for each Banach $\A$-bimodule $X$, and to say $\A$ being approximately contractible is to say $H_s^1(\A,X)=\{0\}$ for each such $X$. The proof of the following lemma is tedious and hence is omitted.

\begin{lemma}\label{lemma}
Let $T\in \Ll^n_s(\A,X)$ ($n\geq 1)$. Suppose that $(\Phi_\al)\subset \Ll^{n-1}_s(\A,X^{**})$ such that $\lim_\al \de^n\Phi_\al = T$ in the strong operator topology. Then there is a net $(\Psi_\be)\subset \Ll^{n-1}_s(\A,X)$ such that $\lim_\be \de^n\Psi_\be = T$ in the weak operator topology.
\end{lemma}

With this lemma we can have the following extension of Theorem~\ref{ac is aa} in the cohomology setting.
\begin{thm}
Let $\A$ be a Banach algebra and $n\geq 1$. If $H_s^n(\A,X^*)=\{0\}$ for each Banach $\A$-bimodule $X$, then $H_s^n(\A,X)=\{0\}$ for each such $X$.
\end{thm}
\begin{proof}
If $H_s^n(\A,X^*)=\{0\}$ for each Banach $\A$-bimodule $X$, then $H_s^n(\A,X^{**})=\{0\}$ for each Banach $\A$-bimodule $X$. Hence for any $T\in \ker(\de_X^{n+1})\subset \ker(\de_{X^{**}}^{n+1})$, there is a net $(\Phi_\al)\subset \Ll_s^{n-1}(\A,X^{**})$ such that $\lim_\al \de^n\Phi_\al = T$ in the strong operator topology. Apply Lemma~\ref{lemma}. We obtain a net $(\Psi_\be)\subset \Ll^{n-1}_s(\A,X)$ such that $\lim_\be \de^n\Psi_\be = T$ in the weak operator topology. Therefore $\ker(\de_X^{n+1})\subset cl_w(\text{Im}\,(\de_X^n))$. From \cite[VI.1.5]{DS}, we have $cl_w(\text{Im}\,(\de_X^n)) = cl_s(\text{Im}\,(\de_X^n))$. So we have shown $\ker(\de_X^{n+1})\subset cl_s(\text{Im}\,(\de_X^n))$. Thus $H_s^n(\A,X)=\{0\}$.
\end{proof}

After Theorems~\ref{ua is a} and \ref{ac is aa}, if we combine equivalent notions in the implication chains (\ref{relation 1}) and (\ref{relation 2}), then the chains are reduced simply to the following chain.
\begin{equation}\label{reduced}
\text{contractible} \Rightarrow\begin{Bmatrix} \text{amenable.}\\
                                      \text{seq. a. c.}\end{Bmatrix}\Rightarrow \begin{Bmatrix} \text{seq. a. a.}\\
                                                                                \text{bdd. a. c.}\end{Bmatrix}\Rightarrow\text{bdd. a. a.}\Rightarrow\text{a. a.} 
\end{equation}
Some partial converses of this reduced chain are true.

\begin{thm}[\cite{GLZ}]
Let $\A$ be a separable Banach algebra. Then $\A$ is sequentially approximately amenable (resp. sequentially approximately contractible) if it is boundedly approximately amenable (resp. boundedly approximately contractible).
\end{thm}

One cannot expect the above result holds without the condition of separability of $\A$. In fact, any amenable Banach algebra without sequential approximate identity is boundedly approximately contractible but not sequentially approximately contractible. Some convolution semigroup algebras are boundedly approximately amenable but not sequentially approximately amenable \cite{CGZ}. We also note that some Feinstein algebras are sequentially approximately contractible but not amenable \cite{GLZ}.
There are two major open questions regarding the converses of the chain (\ref{reduced}).
\begin{question}
 Is there an approximately amenable Banach algebra which is not boundedly approximately amenable?
\end{question}
\begin{question}
 Is there a boundedly approximately amenable Banach algebra which is not boundedly approximately contractible?

\end{question}

It is known that there are pseudo-amenable and even pseudo-contractible Banach algebras which are not approximately amenable. Here are some relations between ``pseudo'' and ``approximate''.

\begin{thm}\label{aa pa} Let $\A$ be a Banach algebra. Then
\begin{enumerate}
\item $\A$ is approximately amenable if and only if $\A^\sharp$ is pseudo-amenable \cite{GZ};
\item $\A$ is boundedly (resp. sequentially) approximately contractible if and only if $\A^\sharp$ is boundedly (resp. sequentially) pseudo-amenable \cite{CGZ};
\item $\A^\sharp$ is pseudo-contractible if and only if $\A$ is contractible \cite{GZ}. (In particular, if $\A$ has a unit, then it is already contractible if it is pseudo-contractible.)
\end{enumerate}
\end{thm}

In general, $\A$ being pseudo-amenable seems much weaker than $\A^\sharp$ being pseudo-amenable (or $\A$ being approximately amenable). But if $\A$ has a bounded approximate identity, then they are equivalent.

\begin{thm}\label{aa psa} Let $\A$ be a Banach algebra.
\begin{enumerate}
\item If $\A$ has a bounded approximate identity, then it is approximately amenable if and only if it is pseudo-amenable \cite{GZ}.
\item $\A$ is boundedly (resp. sequentially) approximately contractible if and only if it is boundedly (resp. sequentially) pseudo-amenable and has a bounded approximate identity.
\end{enumerate}
\end{thm}
\begin{proof}
We prove part (2) for the sequential case. The proof for the other case is given in \cite{GSZ}. If $\A$ is sequentially approximately contractible, then, considering the derivation $a\mapsto a\otimes e-e\otimes a$, we may obtain a sequence $(u_n)\subset \A^\sharp \hot \A^\sharp$ such that $au_n -u_na\to 0$ ($a\in \A$) and $\pi(u_n) = e$. We may write $u_n = v_n - F_n\otimes e -e\otimes G_n + e\otimes e$, where $v_n\in \A\hot \A$, $F_n, G_n \in \A$ and $\pi(v_n) = F_n +G_n$. By the uniform boundedness principle, $(F_n)$ and $(G_n)$ are, respectively, multiplier bounded right approximate identity and multiplier bounded left approximate identity for $\A$. On the other hand, by \cite[Corollary~3.4]{CGZ} (see Theorem~\ref{a.i.}(2) below), $\A$ has a bounded approximate identity $(e_\al)$. This implies that $(F_n)$ and $(G_n)$ are bounded sequences. Let $U_n= v_n - F_n\otimes G_n$. Then it is readily seen that $(U_n)\subset \A\hot\A$ is a sequential approximate diagonal for $\A$. So $\A$ is sequentially pseudo-amenable.

For the converse, assume that $(u_n)\subset \A\hot\A$ is a sequential approximate diagonal for $\A$. Then $(\pi(u_n))$ is a (multiplier bounded) approximate identity for $\A$. It is bounded if $\A$ has a bounded approximate identity. Now define $U_{n,m} = u_n + (e-\pi(u_n))\otimes (e-\pi(u_m))$. Then a subsequence of $(U_{n,m})$ serves a sequential approximate diagonal for $\A^\sharp$. By Theorem~\ref{aa pa}(2), $\A$ is sequentially approximately contractible.
\end{proof}

The existence of a bounded approximate identity in the above theorem cannot be removed. For example, $\ell^1$ is boundedly pseudo-amenable (in fact, it is boundedly pseudo-contractible) but it is not approximately amenable.

\begin{question}
 Does approximate amenability imply pseudo-amenability?
\end{question}

The answer to Question~3 is affirmative if the algebra has a central approximate identity (\cite{GZ}). In particular, it is true if the algebra is abelian. Since every pseudo-amenable Banach algebra has a two-sided approximate identity, Any approximately amenable Banach algebra without a two-sided approximate identity (see Question~4 in Section~ 3) will be a counter-example to this implication conjecture.

Approximate amenability and pseudo-amenability do not imply weak amenability. An example is given in \cite{GL}. But
\begin{thm}[\cite{GZ}]\label{weak}
If $\A$ is an approximate or pseudo amenable abelian Banach algebra, then $\A$ is weakly amenable.
\end{thm}
This result may be useful in studying weak amenability of an abelian Banach algebra.

Recall that a Banach algebra $\A$ is \emph{approximately biprojective} if there is a net $(T_\al)$ of continuous bimodule morphisms from $\A$ into $\A\hot \A$ such that $\lim_\al \pi\circ T_\al (a) = a$ for $a \in \A$ \cite{ZHA}. We have the following relations.

\begin{thm} Let $\A$ be a Banach algebra.
\begin{enumerate}
\item The algebra $\A$ is pseudo-contractible if and only if it is approximate biprojective and has a central approximate identity \cite{GZ}.
\item If $\A$ is approximately biprojective and has an approximate identity, then it is pseudo-amenable.
\end{enumerate}
\end{thm}
\begin{proof}
To prove the second assertion, let $(T_\al)$ be the net of module morphisms described in the definition of approximate biprojectivity, and let $(e_\be)$ be an approximate identity for $\A$. We define $u_{(\al,\be)} = T_\al(e_\be)$. Then one can find a subnet of $(u_{(\al,\be)})$ which form an approximate diagonal for $\A$.
\end{proof}

\section{Some properties of generalized amenability for Banach algebras}

Let $\A$ be a Banach algebra. Let $X$ and $Y$ be left Banach $\A$ modules. Then $B(X,Y)$, the Banach space of all bounded linear operators from $X$ into $Y$ with the uniform norm, is a Banach $\A$ bimodule. The module actions are defined by
\[a\cdot f\,(x) = a(f(x)), \quad f\cdot a \, (x) = f(ax)  \quad (a\in \A,\; x\in X,\; f\in B(X,Y)).  \]

\begin{thm}\label{module}
Let $\A$ be an approximately amenable Banach algebra. Suppose that $f$: $X \to Y$ is a bounded left $\A$ module morphism.
\begin{enumerate}
\item If $f$ has a right inverse $F\in B(Y,X)$, then there is a net $(f_\al)\subset B(Y,X)$ of right inverses of $f$ such that $\|a\cdot f_\al -f_\al \cdot a\| \overset{\al}{\to} 0$ for all $a\in \A$.
\item If $f$ has a left inverse $H\in B(Y,X)$, then there is a net $(h_\al)\subset B(Y,X)$ of left inverses of $f$ such that $\|a\cdot h_\al -h_\al \cdot a\| \overset{\al}{\to} 0$ for all $a\in \A$.
\end{enumerate}
\end{thm}

\begin{proof}
We prove the first assertion. The proof of the second one is similar. If $\A$ is approximately amenable, then $\As$ is pseudo amenable from Theorem~\ref{aa pa}(1). So $\As$ has an approximate diagonal $(u_\al)\subset \As\hot\As$ such that $\pi(u_\al) = e$, where $e$ is the identity of $\As$. We extend $X$ and $Y$ to left $\As$ modules by defining $ex = x$ for $x\in X$ or $Y$. Let $\varPsi$: $\As\hot Y \to X$ be the bounded linear operator specified by
$\varPsi(a\otimes y) = a F(y)$ ($a\in \As$ $y\in Y$). We now define $f_\al$: $Y\to X$ by $f_\al(y) = \varPsi (u_\al y)$ ($y\in Y$). Then clearly $f_\al \in B(Y,X)$ and $f_\al$ is a right inverse of $f$. Moreover 
\[  \|(a\cdot f_\al -f_\al \cdot a)(y)\| = \|\varPsi (au_\al y - u_\al a y)\| \leq \|F\| \|au_\al - u_\al a\|\|y\|  \]
for $a\in \A$ and $y\in Y$. Therefore, $\|a\cdot f_\al -f_\al \cdot a\|\overset{\al}{\to} 0$ ($a\in \A$).
\end{proof}

Using Theorem~\ref{module} directly or using Theorem~\ref{ac is aa}, we can derive an improvement of \cite[Theorem~2.2]{GL} as follows.

\begin{thm}
Suppose that $\A$ is approximately amenable. Let
\[  \Sigma:  0 \to X\overset{f}{\to} Y \overset{g}{\to} Z \to 0  \]
be an admissible short exact sequence of left Banach $\A$ modules. Then $\Sigma$ approximately splits. That is, there is a net $(g_\al)$: $Z\to Y$ of right inverse maps to $g$ such that $\lim_\al (a\cdot g_\al - g_\al \cdot a) =0$ for all $a\in \A$.
\end{thm}

\subsection{Approximate identities}
From the definition it is easy to see that a pseudo-amenable Banach algebra has a two-sided approximate identity, and a pseudo-contractible Banach algebra has a central approximate identity. For approximately amenable Banach algebras we have the following.
\begin{thm}\label{a.i.} Let $\A$ be a Banach algebra.
\begin{enumerate}
\item If $\A$ is approximately amenable, then it has a left and a right approximate identities \cite{GL};
\item If $\A$ is boundedly (resp. sequentially) approximately contractible, then it has a bounded approximate identity (resp. sequential bounded approximate identity) \cite{CGZ}.
\end{enumerate}
\end{thm}

In fact, all known approximately amenable Banach algebras have a bounded approximate identity. However, we do not know this is true in general or not.

\begin{question}\label{tai}
Does every approximately amenable Banach algebra have a two-sided approximate identity? Does it have a bounded approximate identity?
\end{question}

There were some partial results to answer the question in \cite{GLZ}. It is interesting to mention here that if $\A\oplus \A$ is approximately amenable, then $\A$ must have a two-sided approximate identity \cite{GLZ}. So the above question links to Question~\ref{A+B} below. In Fr\'echet algebra setting, an approximately amenable {Fr\'echet} algebra which has no bounded approximate identity was constructed in \cite{LR}.

\begin{question}
 If $\A$ is boundedly approximately amenable, does it have a multiplier-bounded approximate identity?
\end{question}

 If the answer to Question 5 is affirmative, then a boundedly approximately amenable Banach algebra must have a bounded approximate identity \cite[Theorem~3.3]{CGZ}. If the answer is negative, then we may answer Question~2 in the negative by Theorem~\ref{a.i.}(2). 

\subsection{Direct sum and tensor product}
A notable property of pseudo-amenability and pseudo-contractibility is that the two classes are closed under taking $c_0$ and $\ell^p$ direct sums.

\begin{thm}[\cite{GZ}] If $\{\A_\al : \; \al\in \Gamma\}$ is a collection of pseudo-amenable/pseudo-contractible Banach algebras, then $\overset{p} {\oplus} _{\al \in \Gamma}\A_\al$, the $\ell^p$ direct sum of the collection, is pseudo-amenable/pseudo-contractible for any $1\leq p < \infty$ or $p=0$ (here $\ell^0$ means $c_0$).
\end{thm}

Approximate amenability lacks this property. For example, even $\ell^1$, the $\ell^1$-direct sum of $\C$,  is not approximately amenable. But we conjecture that the class of approximately amenable Banach algebras should be closed under taking finite direct sums. This is true for bounded approximate contractibility.
\begin{thm}[\cite{CGZ}]
 If $\A$ and $\B$ are boundedly approximately contractible, then so is $\A\oplus \B$.
\end{thm}

For approximate amenability we only have a partial result.
\begin{thm}[\cite{GLZ}]
 If $\A$ and $\B$ are approximately amenable and one of them has a bounded approximate identity, then $\A\oplus \B$ is approximately amenable.
\end{thm}

\begin{question}\label{A+B}
Is $\A\oplus \B$ approximately amenable if both $\A$ and $\B$ are?
\end{question}

It was shown in \cite{GL} that, if $\A$ is approximately amenable and has a bounded approximate identity and if $\B$ is amenable, then $\A\hot \B$ is approximately amenable. Besides this, little has been known about generalized amenability of tensor products of Banach algebras. Here we have the following.

\begin{thm}
If $\A$ and $\B$ are boundedly pseudo-contractible, then so is $\A\hot \B$.
\end{thm}
\begin{proof}
Let $(u_\al)\subset \A\hot\A$ and $(v_\be)\subset \B\hot\B$ be central approximate diagonals such that $(\pi(u_\al))$ and $(\pi(v_\be))$ are multiplier bounded approximate identities for $\A$ and $\B$, respectively. Suppose 
$u_\al = \sum_i{a^{(\al)}_i\otimes c^{(\al)}_i}$ and $v_\be = \sum_i{b^{(\be)}_i\otimes d^{(\be)}_i}$. Define
\[  U_{(\al,\be)} = \sum_{i,j}{(a^\al_i\otimes b^\be_j)\otimes (c^\al_i\otimes d^\be_j)}  \]
Then $(U_{(\al,\be)})\subset (\A\hot \B)\hot(\A\hot \B)$. One may check that $(U_{(\al,\be)})$ is a central approximate diagonal for $\A \hot \B$ and $\pi(U_{(\al,\be)}) = \pi(u_\al)\otimes \pi(v_\be)$ is a multiplier bounded approximate identity for $\A\hot \B$.
\end{proof}

\begin{question}
Is $\A\hot \B$ approximately amenable (resp. pseudo-amenable) if both $\A$ and $\B$ are?
\end{question}

\subsection{Ideals} 

Most of the hereditary properties asserted in the following theorem are easy to see and can be found in \cite{GL, GZ}.
\begin{thm} Let $\A$ be a Banach algebra and $J$ be a closed ideal of $\A$.
\begin{enumerate}
\item If $\A$ is a. a., b. a. a., seq. a. a., b. a. c., seq. a. c., ps. a., ps. c., b. ps. a., b. ps. c., seq. ps. a. or seq. ps. c., then so is $\A/J$.
\item If $\A$ is a. a., b. a. a., b. a. c., ps. a. or b. ps. a., then so is $J$ if $J$ has a bounded approximate identity.
\item If $\A$ is ps. c. (resp. b. ps. c. or seq. ps. c.), then so is $J$ if $J$ has a (resp. multiplier-bounded or sequential) central approximate identity.
\end{enumerate}
\end{thm}

\begin{question}
If there is a Banach algebra homomorphism $T$: $\A\to\B$ such that $T(\A)$ is dense in $\B$, and if $\A$ is approximately amenable (resp. pseudo-amenable etc.), is $\B$ approximately amenable (resp. pseudo-amenable etc.)?
\end{question}

The existence of approximate identities for ideals of a generalized amenable Banach algebras is an attractive topic.

\begin{thm}[\cite{GSZ}]\label{codi 1}
Let $\A$ be a boundedly approximately contractible. If $J$ is a closed ideal of $\A$ of codimension 1. Then $J$ has a b.a.i..
\end{thm}

We note, unlike amenable case, the above result is false if $J$ is merely a complemented closed ideal of $\A$. A counter-example was given in \cite{GLZ}.

\begin{question}
Does Theorem~\ref{codi 1} still hold if $J$ is a finite codimensional ideal of $\A$?
\end{question}

This is true if $\A$ is abelian, since in the case $A/I$ is a finite dimensional contractible abelian algebra. There are finite 1-codimensional ideals $I_i$ of $\A$, $i=1,2,\cdots, n$, such that $I=\cap_{i=1}^n I_i$.

For approximate or pseudo amenability, we only know some results ensuring one-sided approximate identities for ideals.

\begin{thm} Let $\A$ be a Banach algebra and $J$ be a closed left (resp. right) ideal of $\A$.
\begin{enumerate}
\item If $\A$ is approximately amenable and $J$ is weakly complemented in $\A$, then $J$ has a right (resp. left) approximate identity \cite{GL}.
\item If $\A$ is pseudo-amenable and $J$ is boundedly approximately complemented in $\A$, then $J$ has a right (resp. left) approximate identity; if $\A$ is pseudo-contractible and $J$ is approximately complemented in $\A$, then $J$ has a right (resp. left) approximate identity \cite{GZ}.
\end{enumerate}
\end{thm}

If $J$ is a two-sided ideal of $\A$, then it has both right and left approximate identities under the condition of the above theorem. There is no clue whether it has a two-sided approximate identity.

\begin{question}
Let $\A$ be pseudo-amenable or be approximately amenable with an approximate identity. When does a closed ideal of $\A$ have  a two-sided approximate identity?
\end{question}

\section{Generalized amenability of classical Banach algebras}\label{concrete}

\subsection{Algebras associated to locally compact groups}

There is no difference between generalized amenability and amenability for group algebras.

\begin{thm}[\cite{GL,GZ}]
Let $G$ be a locally compact group. Then
\begin{enumerate}
\item the group algebra $L^1(G)$ is approximately amenable or pseudo-amenable if and only if it is amenable. 
\item the measure algebra $M(G)$ is approximately amenable or pseudo-amenable if and only if $G$ is discrete and amenable.
\item the second dual algebra $L^1(G)^{**}$ is approximately amenable or pseudo-amenable if and only if $G$ is a finite group.
\end{enumerate}
\end{thm}

Consider the Fourier algebras $A(G)$. It is well known that $A(G)$ is not necessarily amenable if $G$ is an amenable group, even $G$ is compact. Amenability of $A(G)$ has recently been characterized in \cite{LLW, FR}. There are pseudo-amenable but not approximate amenable Fourier algebras (e.g. $A(\mathbb F_2)$ \cite{CGZ}), and there are approximately amenable but not amenable Fourier algebras (e.g. on some amenable discrete groups).

\begin{thm}[\cite{GS}]
If $G$ has an open abelian subgroup, then
\begin{enumerate}
\item  $A(G)$ is pseudo-amenable if and only if it has an approximate identity.
\item $A(G)$ is approximately amenable if $G$ is amenable.
\end{enumerate}
\end{thm}

\begin{question}
How to characterize pseudo-amenability and approximate amenability for $A(G)$?
\end{question}
Since pseudo-amenability and approximate amenability both imply weak amenability for $A(G)$, answers to this question may shed light on the investigation of weak amenability of $A(G)$.

A nontrivial Segal algebra is never amenable since it has no bounded approximate identity. In fact, from Theorem~\ref{a.i.}(2) it is never boundedly approximately contractible. But it can be pseudo-amenable and pseudo-contractible.

\begin{thm}[\cite{GZ, CGZ}]
Let $S^1(G)$ be a Segal algebra on a locally compact group $G$.
\begin{enumerate}
\item $S^1(G)$ is pseudo-contractible if and only if $G$ is a compact group.
\item If $S^1(G)$ is pseudo-amenable or approximately amenable, then $G$ is an amenable group.(see also \cite{S-S-S})
\item If $G$ is an amenable SIN-group, then $S^1(G)$ is pseudo-amenable.
\end{enumerate}
\end{thm}

It is unknown whether or not $S^1(G)$ is always pseudo-amenable when $G$ is an amenable group.
It has been pointed out in \cite{CGZ} that a nontrivial symmetric Segal algebra is never boundedly approximately amenable. It was shown in \cite{DLZ} that $\ell^p(E)$ with pointwise multiplication is not approximately amenable if $E$ is an infinite set and $p\geq 1$. This implies that the Segal algebra $L^2(G)$ on an infinite compact abelian group $G$ is not approximately amenable due to the Plancherel Theorem. We also know that the Feichtinger Segal algebra on a compact abelian group is not approximately amenable \cite{CGZ}, some Segal algebras on the circle are not approximately amenable \cite{DLO}, and a nontrivial Segal algebra on $\mathbb R^n$ is not approximately amenable \cite{CG}. All these partial results suggest that the answer to the following question might be possible.

\begin{question}
Is it true that every nontrivial Segal algebra is not approximately amenable?
\end{question}

Let $\omega$ be a continuous weight function on a locally compact group $G$. Let $\Omega(x) = \omega(x)\omega(x^{-1})$ ($x\in G$). N. Gr\o nb\ae k showed in \cite{GRO} that the Beuring algebra $L^1(G, \omega)$ is amenable if and only if $L^1(G, \Omega)$ is amenable if and only if $\Omega$ is bounded and $G$ is an amenable group. Since $L^1(G, \omega)$ has a bounded approximate identity, from Theorem~\ref{ac is aa}, approximate amenability and pseudo-amenability are the same for it. One can see some results regarding generalized amenability of $L^1(G,\omega)$ in \cite{GL, GLZ, GSZ}. So far there is no example of approximately amenable but not amenable Beuring algebras.

\begin{question}
Is it true that $L^1(G,\omega)$ is approximately amenable (pseudo-amenable) if and only if it is amenable?
\end{question}

\subsection{Semigroup algebras}
Let $S$ be a semigroup. Consider the semigroup algebra $\ell^1(S)$. Amenability of $\ell^1(S)$ has recently been characterized in \cite{DLS}. There are boundedly approximately contractible semigroup algebras which are not amenable, and there are pseudo amenable but not approximately amenable semigroup algebras. The study of generalized amenability of semigroup algebras is still far away from completion.

\begin{thm}[\cite{GLZ}]
 If $\ell^1(S)$ is approximately amenable, then $S$ is a regular and amenable semigroup.
\end{thm}

We have known that the bicyclic semigroup $S_1 =< a, b : ab=1>$ is  regular and amenable, but $\ell^1(S_1)$ is not approximately amenable \cite{Gheorghe-Z}.
 Let $\Lambda_\vee$ be the semigroup of a totally ordered set with the product $a\vee b = max\{a,b\}$ ($a,b\in \Lambda_\vee$). the semigroup algebra $\ell^1(\Lambda_\vee )$ is boundedly approximately contractible, but if $\Lambda_\vee$ is an uncountable well-ordered set, then $\ell^1(\Lambda_\vee )$ is not sequentially approximately amenable \cite{CGZ}.
Let $S_b$ be a Brandt semigroup over a group $G$ with an index set $\mathbb I$. Then $\ell^1(S_b)$ is pseudo-amenable if $G$ is amenable; If $\mathbb I$ is infinite, then $\ell^1(S_b)$ is not approximately amenable \cite{Sadr}.

\begin{question}
How to characterize approximate amenability and pseudo-amenability of a semigroup algebra?
\end{question}

\subsection{Other algebras}
Let $H$ be a Hilbert space of infinite dimension. For each $p\geq 1$, it has been shown in \cite{CG} that the Schatten $p$-class algebra $S_p(H)$ is not approximately amenable.
 Let $X$ be an infinite metric space and let $0<\al\leq 1$. Then both Lipschitz algebras $Lip_\al (X)$ and $lip_\al (X)$ are not approximately amenable \cite{GLZ, CG}.
Since $Lip_\al (X)$ and $lip_\al (X)$ are unital, they are not pseudo-amenable.

To the author's knowledge, so far there is no investigation in the literature about generalized amenability of uniform algebras.

\begin{question}
Is there a non-amenable but approximately amenable or pseudo-amenable uniform algebra?
\end{question}

Since uniform algebras are abelian, due to Theorem~\ref{weak}, the above question relates closely to a well-known open question asking whether there is a non-amenable uniform algebra that is weakly amenable.

Let $X$ be a Banach space. Denote by $\K(X)$ the Banach algebra of compact operators on $X$ with the composition multiplication and the operator norm topology. We wonder if there is an $X$ such that $\K(X)$ is not amenable but approximately amenable or pseudo-amenable. In general, the following question is open.

\begin{question}
When is $\K(X)$ approximately amenable? When is it pseudo-amenable?
\end{question}

Let $G$ be a discrete group. The reduced group C* algebra $C_r^*(G)$ (and the full group C* algebra $C^*(G)$) is approximately amenable if and only if it is amenable \cite{CGZ}. From Theorem \ref{aa psa}, for a general C*-algebra, approximate amenability is the same as pseudo-amenability since it has a bounded approximate identity. We end the paper with the following question.

\begin{question}
Is there an approximately amenable but not amenable C*-algebra? If yes, how to characterize approximate amenability for a C* algebra?
\end{question}

This paper is based on a lecture delivered at the 19$^\mathrm{th}$ 
International Conference on Banach Algebras held at B\c{e}dlewo, July 14--24, 
2009. The support for the meeting by the Polish Academy of Sciences, the 
European Science Foundation under the ESF-EMS-ERCOM partnership, and the 
Faculty of Mathematics and Computer Science of the Adam Mickiewicz University 
at Pozna\'n is gratefully acknowledged.


\begin{thebibliography}{40}


\bibitem{BCD} W.G. Bade, P, C, Curtis Jr. and H.G. Dales,
\emph{Amenability and weak amenability for Beurling and Lipschitz algebras}, 
{Proc. London Math. Soc.} {55} (1987), 359--377.


\bibitem{CAL} P.C. Curtis Jr. and R.J. Loy, \emph{The structure of amenable Banach algebras},
{J. London Math. Soc.} {40} (1989), 89--104.

\bibitem{CG}
Y. Choi and F. Ghahramani,
\emph{Approximate amenability of Schatten classes, Lipschitz algebras and second duals of Fourier algebras}, preprint, arXiv 0906.2253 (2009).

\bibitem{CGZ}
Y. Choi, F. Ghahramani and Y. Zhang,
\emph{Approximate and pseudo-amenability of various classes of Banach algebras}, J. Funct. Anal. {256} (2009), 3158--3191.



\bibitem{Con}
A. Connes,
\emph{On the cohomology of operator algebras}, J. Funct. Anal. 28 
(1978), 248--253.

\bibitem{DAL} H. G. Dales, \emph{Banach algebras and automatic continuity}, Clarendon Press,
Oxford, (2000).

\bibitem{D-G-G}
  H. G. Dales, F. Ghahramani and N. Gr\o nb\ae k, 
  \emph{Derivations into iterated duals of Banach algebras},
  Studia Math. {128} (1998), 19--54.

\bibitem{DL}
H. G. Dales and A. T.-M. Lau,
\emph{The second duals of Beurling algebras}, Mem. Amer. Math. Soc. 177 (2005).

\bibitem{DLS}
H. G. Dales, A. T.-M. Lau and D. Strauss,
\emph{Banach algebras on semigroups and on their compactifications}, Mem. Amer. Math. Soc., to appear.

\bibitem{DLO}
H. G. Dales and R. J. Loy, \emph{Segal algebras on R and T} (title needs to be confirmed), preprint.

\bibitem{DLZ} H. G. Dales, R. J. Loy and Y. Zhang, \emph{Approximate amenability for 
Banach sequence algebras}, Studia Math. 177 (2006), 81--96. 





\bibitem{DS}
N. Dunford and J. T. Schwartz,
\emph{Linear operators.  Part I. General theory}, John Wiley \& Sons, New York, (1988).


\bibitem{FR}
B. E. Forrest and V. Runde,
\emph{Amenability and weak amenability of the Fourier algebra}, Math. Z. 250 (2005), 731--744. 





\bibitem{GL} { F. Ghahramani and R. J. Loy},\emph{ Generalized notions of amenability},
{J. Funct. Anal.} {208} (2004), 229--260.

\bibitem{GLZ}
F. Ghahramani, R. J. Loy and Y. Zhang,
\emph{Generalized notions of amenability, II}, J. Funct. Anal. {254} (2008), 1776-1810.

\bibitem{GSZ}
F. Ghahramani, E. Samei and Y. Zhang,
\emph{Generalized amenability of Beurling algebras}, preprint.


\bibitem{GS} { F.\ Ghahramani and R.\ Stokke}, \emph{Approximate and Pseudo-Amenability of $A(G)$}, Indiana Univ. Math. J. 56 (2007), 909--930. 

\bibitem{GZ} { F. Ghahramani and Y. Zhang}, \emph{Pseudo-amenable and pseudo-contractible Banach algebras}, Math. Proc. Cambridge Philos. Soc. 142 (2007), 111--123.


\bibitem{Gheorghe-Z}
  F. Gheorghe and Y. Zhang,
  \emph{A note on the approximate amenability of semigroup algebras}, Semigroup Forum 79 (2009), 349--354.

\bibitem{GOU} {F. Gourdeau}, \emph{ Amenability of Lipschitz algebras}, {Math. Proc.
Cambridge Philos. Soc.} {112} (1992), 581--588.

\bibitem{GRO} {N. Gr\o nb\ae k}, \emph{ Amenability of weighted convolution algebras 
on locally compact groups}, {Trans. Amer. Math. Soc.} {319} 
(1990), 765--775.


\bibitem{haage}
U. Haagerup,
\emph{All nuclear C*-algebras are amenable}, Invent. Math. 74 (1983), 
305--319.



\bibitem{HEL} {A.Ya. Helemskii}, \emph{Banach and locally convex algebras},
 Oxford Press, Oxford, (1993).
 
 \bibitem{HEL_C} {A.Ya. Helemskii}, \emph{Homological essence of amenability in the sense of A. Connes: the injectivity of the predule bimodule},
 Mat. Sb. 180 (1989), 1680--1690, 1728; Math. USSR-Sb. 68 (1991), 555--566.
 
\bibitem{Johns_M}{B.E. Johnson},\emph{ Cohomology in Banach algebras},
{ Mem. Amer. Math. Soc.} 127 (1972).

\bibitem{John_diag} {B.E. Johnson}, \emph{Approximate diagonals and cohomology of certain annihilator Banach
algebras}, {Amer. J. Math.}  94 (1972), 685--698.


\bibitem{JKR}
B. E. Johnson, R. V. Kadison and J Ringrose
\emph{Cohomology of operator algebras. III. Reduction to normal cohomology}, Bull. Soc. Math. France {100} (1972), 73--96.

\bibitem{LLW}
A. T.-M. Lau, R. J. Loy and G. A. Willis,
\emph{Amenability of Banach and $C\sp *$-algebras on locally compact groups}, Studia Math. {119} (1996), 161--178.

\bibitem{LR}
P. Lawson and C. J. Read,
\emph{Approximate amenability of Fr\'echet algebras}, Math. Proc. Cambridge Philos. Soc. 145 (2008), 403--418.





\bibitem{Pir_approj}
A.~{\relax Yu}. Pirkovskii, \emph{Approximate characterizations of projectivity and
  injectivity for {B}anach modules}, Math. Proc. Cambridge. Philos. Soc. 143
 (2007), 375--385.

\bibitem{Ruan}
Z. J. Ruan, \emph{The operator amenability of $A (G)$}, Amer. J. Math. 117 (1995), 1449--1476.


\bibitem{RUN} V. Runde, 
\emph{Connes-amenability and normal, virtual diagonals for measure algebras},
J. London Math. Soc. 67 (2003), 643--656.

\bibitem{Sadr}
M. M. Sadr,
\emph{Pseudo-amenability of Brandt semigroup algebras}, preprint.

\bibitem{S-S-S}
E.~Samei, R.~Stokke, N.~Spronk, \emph{Biflatness and pseudo-amenability of {S}egal
  algebras}, preprint, arXiv {\tt 0801.0731} (2008).

\bibitem{se_uniform}
M. V. \v Se\u\i nberg,
 \emph{On a characterization of the algebra $C(\Omega)$ in terms of cohomology groups}, Uspekhi Mat. Nauk 32 (1977), 203--204.
 




\bibitem{ZHA} {Y. Zhang}, \emph{ Nilpotent ideals in a class of Banach algebras}, {Proc. Amer.
Math. Soc.} { 27} (1999), 3237--3242.


\end{thebibliography}
\end{document}